\documentclass[11pt]{amsart}
\usepackage{amscd,amssymb,amsmath,enumerate}
\usepackage{amsmath,amsthm}
\usepackage[all]{xy}
\usepackage{graphicx}
\usepackage{graphics}
\usepackage{latexsym}
\usepackage[all]{xy}
\usepackage{psfrag}
\xyoption{matrix} \xyoption{arrow}
\usepackage{parskip}

\newtheorem{proposition}{Proposition}[section]
\newtheorem{theorem}[proposition]{Theorem}
\newtheorem{lemma}[proposition]{Lemma}
\newtheorem{corollary}[proposition]{Corollary}

\newtheorem{definition}[proposition]{{Definition}}

\newtheorem{remark}[proposition]{{Remark}}

\newtheorem{Question}[proposition]{Question}
\newtheorem{Example}[proposition]{Example}

\newcommand{\cA}{{\mathcal A}}
\newcommand{\cB}{{\mathcal B}}

\newcommand{\cI}{{\mathcal I}}
\newcommand{\cN}{{\mathcal N}}

\newcommand{\cX}{{\mathcal X}}

\newcommand{\cQ}{{Q}}

\newcommand{\cH}{{\mathcal H}}
\newcommand{\cT}{{\mathcal T}}
\newcommand{\cG}{{\mathcal G}}

\newcommand{\cC}{{\mathcal C}}

\newcommand{\cS}{{\mathcal S}}

\newcommand{\clU}{{\mathcal U}}
\newcommand{\K}{{ K}}

\newcommand{\Ext}{\operatorname{Ext}\nolimits}

\newcommand{\pd}{\operatorname{pd}\nolimits}
\renewcommand{\dim}{\operatorname{dim}\nolimits}
\newcommand{\gldim}{\operatorname{gl.dim}\nolimits}

\newcommand{\grb}{Gr\"obner\ }
\newcommand{\tip}{\operatorname{tip}}

\newcommand{\nontip}{\operatorname{nontip}}

\newcommand{\Span}{\operatorname{Span}}

\newcommand{\id}{\operatorname{id}\nolimits}

\newcommand{\GrAlg}{\operatorname{GrAlg}\nolimits}

\makeatletter
\def\thm@space@setup{%
  \thm@preskip=0.7cM \thm@postskip=0.3cM
}
\makeatother

\begin{document}

\date{today}
\title[Strong Koszul Algebras]
{The geometry of strong Koszul algebras}

\author[Green]{Edward L.\ Green}
\address{Edward L.\ Green, Department of
Mathematics\\ Virginia Tech\\ Blacksburg, VA 24061\\
USA}
\email{green@math.vt.edu}

\subjclass[2010]
{16S37, 14M99,16W60}

\keywords{}
\thanks{}

\begin{abstract}Koszul algebras with quadratic \grb bases, called strong
Koszul algebras, are studied.  We introduce
affine algebraic varieties whose points are in one-to-one correspondence with
certain strong Koszul algebras and we investigate the connection between the
varieties and the algebras. 
\end{abstract}
\date{\today}
\maketitle

\section{Introduction}\label{sec-ntro}The connection between affine algebraic
varieties and commutative rings, especially quotients of commutative polynomial
rings over a field, is well established.  In this paper, we introduce a new connection
between affine algebraic varieties and a class of Koszul algebra which are not 
necessarily commutative.  The varieties we consider have the property that the
points  are in one-to-one correspondence
with certain Koszul algebras.   Given one of these varieties, the Koszul algebras corresponding to
the points are shown to have a number of features in common.

To describe our results more precisely, let $K$ be a field, $\cQ$ a finite quiver, and let
$K\cQ$ denote the path algebra.  An element $x\in K\cQ$ is called  \emph{quadratic} if
$x$ is a $K$-linear combination of paths of length 2.  
We say a (two-sided) ideal $I$ in $K\cQ$ is a \emph{quadratic ideal} if $I$ can be generated
by quadratic elements.  One of many equivalent definitions of a \emph{Koszul algebra} is
the following:  If $J$ is the ideal in $K\cQ$ generated by the arrows of $\cQ$ and
$I$ is a quadratic ideal in $K\cQ$ , then $K\cQ/I$ is a \emph{Koszul algebra} if the Ext-algebra,
$\oplus_{n\ge 0}\Ext^n_{K\cQ/I}(K\cQ/J,K\cQ/J)$, can be
 generated in degrees 0 and 1.  Although this a very special class of algebras, Koszul algebras 
occur in many different settings, for example, see \cite{C,F,Gr1} and their references.

In this paper we study a class of Koszul algebras which we call \emph{strong Koszul algebras}.
To define this class, fix an admissible order $\succ$ on the paths in $\cQ$.  The formal definition of an admissible order
can be found in the beginning of Section \ref{sec-strong}. Such an 
order is needed for $K\cQ$ to have a \grb basis theory.  We provide a brief overview of
the \grb basis theory needed for the paper in Section \ref{sec-strong}.
An algebra $\Lambda=K\cQ/I$ is a \emph{strong Koszul algebra (with respect to $ I$ and
$\succ$)}, if $I$ is a quadratic ideal in  $K\cQ$ 
and $I$ has a \grb basis consisting of quadratic elements.  That a strong Koszul algebra
is in fact a Koszul algebra is proved in \cite{GH}.  Note that not all Koszul algebras are strong; for example, Sklyanin algebras \cite{S} are Koszul algebras that are not strong.    
One important type of a strong Koszul algebra
is of the form  $K\cQ/I^*$, where $I^*$ is an ideal that 
can be generated by a set $\cT$ of paths of length 2 in $\cQ$.
That $K\cQ/I^*$ is a strong Koszul algebra follows from the fact that $\cT$ is a \grb
basis of $I^*$ with respect to any admissible order \cite{GH}.

For each set $\cT$ of paths of length 2 in $K\cQ$, we define an affine algebraic variety $\GrAlg(\cT)$,  such that the points of $\GrAlg(\cT)$ are in one-to-one correspondence
with a particular set of strong Koszul algebras; see Theorem \ref{thm-one} and
Theorem \ref{thm-var}.  We view this correspondence as an identification and show that
each strong Koszul algebra corresponds to a point in one of these varieties; see Corollary \ref{cor-corres}.
In each variety $\GrAlg(\cT)$, there is a distinguished algebra, $K\cQ/I^*$, where $I^* $ is generated by $\cT$.  All other algebras $K\cQ/I$ in $\GrAlg(\cT)$ have the property that $ I$ cannot be generated
by paths.

The class of strong Koszul algebras includes preprojective algebras whose underlying graph is connected
and not a tree \cite{Gr1}, straightening closed algebras generated by minors \cite{GH}, and algebras of the form
$K\cQ/\langle \cT\rangle$ where $\cT$ is a set of paths of length 2 in $K \cQ$ and 
$\langle \cT \rangle $ denotes the ideal generated by  $\cT$.

The algebras
lying in one variety have a number of properties in common.  Suppose that
$\Lambda=K\cQ/I$ and $\Lambda'=K\cQ/I'$ are two strong Koszul algebras in
$\GrAlg(\cT)$. Then we prove the following results. Let $\Lambda^*=K\cQ/\langle \cT\rangle$.
\begin{enumerate}
\item $\dim_K(\Lambda)= \dim_K(\Lambda')= \dim_K(\Lambda^*)$; see Theorem \ref{thm-cartan}.
\item Assuming $\Lambda^*$ is finite dimensional,  then $\gldim(\Lambda)=\gldim(\Lambda')=\gldim(\Lambda^*)$; see Corollary \ref{cor-gldim}. Note that 
\cite{GHZ} provides
a fast algorithm for computing the global dimension of $\Lambda^*$.
\item The Betti  numbers in the minimal projective resolutions of one dimensional simple
modules for all three algebras are the same; see Theorem \ref{thm-ext}.
\item The Cartan matrices of $\Lambda$, $\Lambda'$, and $\Lambda^*$ are the same; see Corollary
\ref{cor-cartan}.
\item If $\Lambda^*$ is quasi-hereditary, then so are $\Lambda$ and $\Lambda'$ \cite{GHS}.  Furthermore
a method for determining if $\Lambda^*$ is quasi-hereditary is given in \cite{GHS}.
\end{enumerate} 
   
Section \ref{sec-ex} is devoted to examples.  In particular, the varieties that include the
commutative polynomial rings of dimension 2 and 3 are investigated; see Examples \ref{ex-poly2}
and \ref{ex-poly3}.
 These examples show that in certain cases, the algebras in the variety $\GrAlg(\cT)$ are 
Koszul Artin-Schelter regular algebras, which
have played a fundamental role in  noncommutative geometry; see, for example, 
\cite{HOZ,MS} and their references.

Section \ref{sec-sub} shows how to restrict the varieties to subvarieties
that can be more tractible than the full variety $\GrAlg(\cT)$.  Section \ref{sec-results} shows that
if $\Lambda$ is a strong Koszul algebra, then so are the opposite algebra $\Lambda^{op}$ and 
the enveloping algebra, $\Lambda\otimes_K\Lambda^{op}$.  The paper ends with some
remarks and questions.

\section{Strong Koszul algebras}\label{sec-strong}
To define a strong Koszul algebra we will need to briefly
review (graded) \grb basis theory. For details we refer the
reader to \cite{Gr2}.  We fix a field $K$ and
a finite quiver $\cQ$.  The set $\cB$   of finite (directed) paths 
forms a $K$-basis of the  path algebra $K\cQ$.  We
positively $\mathbb  Z$-grade $K\cQ=K\cQ_0+K\cQ_1+\cdots$ by defining
$K\cQ_n$ to be the $K$-span of paths in $ \cB$ of length $n$.  This is called
the \emph{length grading} on $K\cQ$.

For a \grb basis theory 
we need a special type of order on $\cB$.   We say a well-order $\succ$ on $\cB$
is \emph{admissible} if, for all $p,q,r,s,t$ in $\cB$,
\begin{enumerate}
\item if $p\succ q$, then $pr\succ qr$ if both $pr$ and $qr$ are nonzero,
\item if $p\succ q$, then $sp\succ sq$ if both $sp$ and $sq$ are nonzero, and
\item if $p=rqt$, then $p\succeq q$.
\end{enumerate}
We fix an admissible order $\succ$ on $\cB$.  Since we are interested in
graded Koszul algebras where the grading is induced from the length
grading of $K\cQ$, we add the requirement that if $p,q\in\cB$ and $\ell( p) >\ell (q)$
then $p\succ q$, where $\ell(p)$ denotes the path length of $p$.  
We call such an admissible order a \emph{length admissible
order}.  Fix a length admissible order $\succ$.  Note that   we give an example of
such an order in the beginning of Section \ref{sec-ex}.

In general, $\cB$ will be an infinite set.  We make the convention that if $x\in K\cQ$ 
and
we write  $x=\sum_{p\in\cB}\alpha_pp$ with $\alpha_p\in K$, then all but a finite
number of $\alpha_p$ eqal 0.
If $x=\sum_{p\in\cB}\alpha_pp$ is a nonzero
element of $K\cQ$, then $\tip(x)=p$ if $\alpha_p\ne 0$ and $p\succeq q$, for
all $q$ with $\alpha_q\ne 0$.  If $X\subseteq K\cQ$, then
\[\tip(X)=\{\tip(x)\mid x\in X \text{ and }x\ne 0\}.\]

We say a nonzero element $x\in K\cQ$ is \emph{uniform} if there exist vertices
$v$ and $w$ such that $x=vxw$.  Paths are always uniform,
and, if $\cQ$ has one vertex and $n$ loops,
then $K\cQ$ is isomorphic to the free algebra on $n$ noncommuting variables and
that every nonzero element of $KQ$ is uniform.

If $I$ is an ideal in $K\cQ$, then we say that $I$ is 
a \emph{graded ideal} if $I =\sum_{n\ge 0}I\cap K\cQ_n$.  Equivalently,
$I$ can be generated by (length) homogeneous elements.  If $I$ is a graded
ideal in $K\cQ$ and $\Lambda=K\cQ/I$, then $\Lambda$ has positive $\mathbb Z$-grading
induced from the length grading on $K\cQ$, which we call the \emph{induced length grading}.

\begin{definition}\label{def-gb}{\rm
Let $I$ be a graded ideal in $K\cQ$ and $\cG$ a set of length homogeneous
uniform elements in $I$.  Then $\cG$ is a \emph{(graded) \grb basis} of $I$ (with respect
to $\succ$) if 
\[
\langle \tip(\cG)\rangle=\langle\tip(I)\rangle.\]

}
\end{definition}

\begin{definition}\label{def-st-kosz}{\rm Let $\Lambda =K\cQ/I$.  We say that $\Lambda$ is a \emph{strong
Koszul algebra} (with respect to $I$ and $\succ$) if $I$ has a  \grb basis with respect
to $\succ$ consisting  of quadratic elements.
 }
\end{definition}

\begin{theorem}\label{thm-is-kz}\cite{GH} A strong Koszul algebra is a Koszul algebra.
\end{theorem}

The converse is false in general, for example, the Sklyanin algebras \cite{S}.

For the remainder of this section we look more closely at the strong Koszul algebras.  If $X$ is
a subset of $K\cQ$, then we define
\[\nontip(X)=\cB\setminus \tip(X)\]

We have the following result whose proof is left to the reader.

\begin{proposition}\label{prop-tip-nontip}.  Let $\cT$ be a set of paths in $K\cQ$ such that
if $t,t'\in\cT$ and $t\ne t'$ then $t$ is not a subpath of $t'$.  Let $I=\langle \cT\rangle$.
\begin{enumerate}
\item $n\in\nontip(I)$ if and only if no path in $\cT$ is a subpath of $n$.
\item $t=a_1a_2\cdots a_n$ with $a_i\in\cQ_1$ is in $\cT$ if and only if $t\notin\nontip(I)$ but
$a_2a_3\cdots a_n$ and  $a_1a_2\cdots a_{n-1}$ are in $\nontip(I)$.
\end{enumerate}\qed\end{proposition}

Our next result is fundamental and is slightly more
general than the result found in  \cite{Gr2}. 
Let  $\cS$ denote the subalgebra of $K\cQ$ generated
by the vertices of $\cQ$.  Note that $\cS$ is a semisimple $K$-algebra.  
If $\cX$ is a set of paths in $\cQ$,  $\Span_K(\cX)$ is an $S$-bimodule
as follows: if $\sum_{x\in\cX}\alpha_xx\in \Span_K(\cX)$ and $v,w\in\cQ_0$,
then $v(\sum_{x\in\cX}\alpha_xx)w=\sum_{x\in\cX}\alpha_x(vxw)$.
The proof found in \cite{Gr2} can easily be adjusted from $K$-vector spaces to
$S$-bimodules.

\begin{lemma}\label{lem-fund}{\rm\bf (Fundamental Lemma)}
\ If $I$ is an ideal in $K\cQ$, then
\[ K\cQ = I\oplus \Span_K(\nontip(I)),\]
as $S$-bimodules.

\end{lemma}

We say an ideal $I$ in $K\cQ$ is a  \emph{monomial} ideal if $I$ can be generated
by paths.  Note that a monomial ideal is a  graded ideal. The proof of the following
well-known result is left to the reader.

\begin{proposition}\label{prop-mono}Let $L$ be a monomial ideal in $K\cQ$.
Then
\begin{enumerate}
\item an element $x=\sum_{p\in\cB}\alpha_pp$ with $\alpha_p\in K$ is in $L$ if and
only if, for each $\alpha_p\ne 0$, $p\in L$, and
\item there is a unique minimal set of  paths that generate $L$.\qed
\end{enumerate}
\end{proposition}

We apply the second part of the above proposition and the Fundamental Lemma as follows.
Let $I$ be a graded ideal in $K\cQ$.  Then $\langle \tip(I)\rangle$, the two-sided ideal generated
by $\tip(I)$, is a monomial ideal. Hence there is a unique minimal subset, $\cT$, of $\tip(I)$,
that generates $\langle \tip(I)\rangle$.  By the Fundamental Lemma, for each $t\in\cT$,
there is a unique $g_t\in I$ and a unique $n(t)\in\Span_K(\nontip(I)$ such that
$t=g_t+n(t)$.  In particular, for each $t\in\cT$, $t-n_t\in I$.

\begin{proposition}\label{prop-red}  The set $\cG=\{g_t\mid t\in \cT\}$ is a graded \grb basis
for $I$.
\end{proposition}
\begin{proof}
 Each $g_t\in I$ implies that $\tip(g_t)\in\tip(I)$.  Since $n(t)\in \Span_K(\nontip(I))$, we
conclude that, for each $t\in\cT$, $\tip(g_t)=t$.  Next we show that $\cG$
consists of uniform length homogeneous elements.  Letting $t\in\cT$ be a path of length $m$,
writing $t=g_t+n(t)$ we see that, in degree $m$, $(g_t)_m\in I$ and $n(t)_m$ remains in
$\Span_K(\nontip(I))$.  We have that $t=(g_t)_m+n(t)_m$ and, by unicity, each $g_t$ is a length homogeneous element.  The proof that each $g_t$ is uniform is similar.
Since $\cT$ generates $\langle\tip(I)\rangle$, $\cT=\tip(\cG)$, the elements of $\cG$ 
are uniform, length homogenous, and hence we  are done.
\end{proof}

\begin{definition}\label{def-red}{\rm Given a graded ideal $I$ in $K\cQ$ and $\cG$ as constructed
above, we call $\cG$ the \emph{reduced} \grb basis for $I$ (with respect to $\succ$).}
\end{definition}

Returning to strong Koszul algebras, if an ideal has a 
\grb basis $\cH$ of uniform quadratic
elements then $\tip(\cH)$ consists of paths of length 2, and hence the reduced \grb basis
consists of quadratic uniform elements.  Thus, 
\begin{proposition}\label{prop-ska}We  have that 
$\Lambda=K\cQ/I$ is a strong Koszul algebra if and only
the reduced \grb basis consists of quadratic uniform elements.
\end{proposition}

\section{The variety $\GrAlg(\cT)$}\label{sec-var}
In this section, if $\cT$ is a set of paths of length 2, we define an affine variety
whose points are in one-to-one correspondence to the strong Koszul algebras $\Lambda=K\cQ/I$
(with respect to $I$ and $\succ$) ,
having the property that $\langle\tip(I)\rangle$ is generated by $\cT$.  Referring to Example
\ref{ex-gd} while reading this section
should be helpful.

 Fix $\cT$ to be a set of paths of length 2.
Set $\cN=\cB\setminus \tip( \langle\cT\rangle)$.   Recall that if $I $ is an ideal such
that $\langle \tip(I)\rangle =\langle \cT\rangle$, then $\cN=\nontip(I)$.  It is important
to note that $\cN$ is only dependent on $\cT$ and not on $I$.   

We begin by defining the affine space in which our variety lives.  For this we need the following
definitions.  
We say two elements $x$ and $y$ of $K\cQ$  are    \emph{parallel} if there are vertices $v$ and $w$ such
that $vxw=x$ and $vyw=y$.  In particular, if $x$ and $y$ are parallel then both $x$ and
$y$ are uniform. 
If $x$ and $y$ are parallel, we write $x\| y$. Note that if $x=\sum_{p\in\cB}\alpha_pp\in K\cQ$, then $x$ is
uniform if and only if $p\| q$, for all $p,q\in\cB$ with $\alpha_p$ and $\alpha_q$ 
nonzero.

Let $\cN_2$ be the set of paths of length 2 in $\cN=\cB\setminus \tip(\langle\cT\rangle)$ and, for $t\in\cT$, define
\[ \cN_2(t)=\{n\in\cN_2\mid t\succ n\text{ and } n\| t\}.\] 

We now can define the affine space in which our variety lives.
If $S$ is a set, then $|S|$ denotes the cardinality of $S$.  Let $D=\sum_{t\in\cT}|\cN_2(t)|$.
We let $\cA=K^D$, viewed as a $D$-dimensional affine space.   If $X\in \cA$, then
we write $X$ as tuple with indices in the disjoint union of
the $\cN_2(t)$'s; that is, we write $X=(x_{t,n})$, where $t\in \cT$, $
n\in\cN_2(t)$, and $x_{t,n}\in K$.   

For each $X=(x_{t,n})\in\cA$, let \[\cG(X)=
\{g_t\in K\cQ\mid g_t=t-\sum_{n\in\cN_2(t)}x_{t,n}n\}.\]

We now define the subset of $\cA$ of interest.
\begin{definition}{\rm Given a set $\cT$ of paths of length $2$, define
\[ \GrAlg(\cT)=\{X\in\cA\mid K\cQ/\langle \cG(X)\rangle \text{ is a strong Koszul
algebra (with}\]\[\text{ respect to }\langle \cG(X)\rangle \text{ and }\succ)
\}
\]
}
\end{definition}
The remainder of this section is devoted to showing that $\GrAlg(\cT)$ is an affine
variety in $\cA$ whose points are in one-to-one correspondence with the elements of
\[\clU =\text{ the set of
algebras }K\cQ/I\text{ that are strong Koszul algebras}\]\[\text{ (with respect to }I\text{ and }\succ)\text{ such
that }\langle \cT\rangle=\langle \tip(I)\rangle.\]First we will show the one-to-one correspondence.  We begin with a   preparatory
result.

\begin{proposition}\label{prop-oneone}  If  $K\cQ/I\in\clU$ then there exists $X\in\cA$
such that the reduced \grb basis of $I$ is $\cG(X)$ for some $X$.   Moreover, $X$ is
unique.
\end{proposition}

\begin{proof}   Suppose that $K\cQ/I\in\clU$.  Since $\langle \cT\rangle=\langle \tip(I)\rangle$
and $\cT$ are paths of length 2, $\cT$ must be the unique minimal generating set of the
monomial ideal $\langle \tip(I)\rangle$.  It now follows from our discussion of the reduced
\grb basis that there is some $X\in\cA$ such that $I=\langle \cG(X)\rangle$.  Uniqueness follows
from the uniqueness of the reduced \grb basis.
\end{proof}

\begin{corollary}\label{cor-corres} If $\Lambda=K\cQ/I$ is a strong Koszul algebra
(with respect to $I$ and  $\succ$), then $\Lambda\in\GrAlg(\cT)$ where
$\cT$ is the minimal set of generators of $\langle \tip(I)\rangle$.
\end{corollary}
\begin{proof} The reduced \grb basis $\cG$ of $I$ with respect to $\succ$ is composed
of uniform quadratic elements.   Let $\cT=\tip(\cG)$.  It is immediate that
$\cT$ is the minimal set of paths that generate $\langle \tip(I)\rangle$.  It is now clear that
$\Lambda\in   \GrAlg(\cT)$.
\end{proof}
We now state the correspondence theorem.

\begin{theorem}\label{thm-one}Let $\cT$ be a set of paths of length 2.
There is a one-to-one correspondence between the
points of $\GrAlg(\cT)$ and  the
algebras $K\cQ/I$ that are strong Koszul algebras (with respect to $I$ and $\succ$) such
that $\langle \cT\rangle=\langle \tip(I)\rangle$.
\end{theorem}

\begin{proof}
Define $\varphi\colon \GrAlg(\cT)\to \clU$ by $\varphi(X)=KQ/\langle \cG(X)\rangle$.  The
map $\varphi$ is well-defined. We see that $\varphi$ is injective, since if 
$X=(x_{t,n}),X'=(x'_{t,n})\in\GrAlg(\cT)$
with $X\ne X'$, the reduced \grb bases of $\varphi(X)$ and $\varphi(X')$ differ.  But
the reduced  \grb basis of an ideal is unique, and $\varphi$ being
injective follows.

To see that $\varphi$ is onto, let $K\cQ/I\in\clU$.  We are assuming that
$\langle \tip(I)\rangle=\langle \cT\rangle$.  Since every path in $\cT$ is
a path of length 2 and $\langle \tip(I)\rangle=\langle\cT\rangle$, 
we see that $\cT$ is the (unique) minimal set of  paths
that generate $\langle \tip(I)\rangle$.  By the construction of the reduced
\grb basis for $I$ with respect to $\succ$ found in Section \ref{sec-strong}, the
reduced \grb basis for $I$ with respect to $\succ$ is
$\{g_t\mid t\in\cT \text{ and }g_t =t-\sum_{n\in\cN_2(\cT)}x_{t,n}n\}$ for
some $x_{t,n}\in K$.  Thus, $K\cQ/I=\varphi((x_{t,n}))$, and we are done.

\end{proof}

We now show the somewhat surprising result that $\GrAlg(\cT)$ is an affine algebraic variety  in $\cA$.

\begin{theorem}\label{thm-var} Let $K$ be a field, $\cQ$ a finite quiver, and $\cT$ be a
set of paths of length 2 in $\cQ$.  Then
$\GrAlg(\cT)$ is an affine algebraic variety.
\end{theorem} 

Before proving this result, we will need some preliminary work.

We introduce a ``polynomial ring over a path algebra".    Let $\mathbf y$ be a
set of $D$ variables with $\{y_{t,n}\}=\mathbf y$ where $t\in \cT$ and $n\in\cN_2(t)$.
Consider the ring $R=K\cQ[\mathbf y]$, consisting of finite sums
of the form $\sum_{p\in\cB}f_p(\mathbf y)p$, where $f_p(\mathbf y)$ is a polynomial in
the commutative polynomial ring $K[\mathbf y]$.  The variables $ y_{t,n}$ commute with
elements   of the path algebra $K\cQ$ (and each other).  Note that $K[\mathbf y]$ is the coordinate ring
of the affine space $\cA$. Given an element $\sum_{p\in\cB}f_p(\mathbf y)p\in R$,  we call
the polynomial $f_p(\mathbf y)$ the `coefficient' of  $p$

We are interested in a particular set of elements in $R$, namely
\[\cH=\{h_t\in R\mid  h_t=t-\sum_{n\in\cN_2(t)}y_{t,n}n\}.\]
  If  $F=\sum_{p\in\cB}f_p(\mathbf y)p$ is an element
of $R$, then we say $F'\in R$ is a \emph{simple reduction of $F$ by $\cH$}, written $F\to_{\cH}F'$, if
there is some $p\in\cB$ and $t\in \cT$ such that
\begin{enumerate}
\item $p=qtr$ for some paths $q$ and $r$,
\item $f_p(\mathbf  y)\ne 0$, and
\item $F'= F-f_p(\mathbf y)p  + (f_p(\mathbf y)(q(\sum_{n\in\cN_2(t)}y_{n,t}n)r)$.
\end{enumerate}

The effect of a simple reduction is the following.  Suppose $f_p(\mathbf y)p$ occurs in $F$ 
with $p=qtr$ and
it is the term we work with.  Then the term $f_p(\mathbf y)p$ is replaced with
the sum of terms $(f_p(\mathbf y)y_{n,t})qtr$, for $n\in\cN_2(t)$.  Note that  $p\succ qtr$
for each $n\in\cN_2(t)$.  Thus, for each $n\in\cN_2(t)$, $f_p(\mathbf y)y_{n,t}$ is added to 
$f_{qtr}(\mathbf y)$ as the 'coefficient' in front of $p=qtr$ and $f_p(\mathbf y)p$ is removed.
All other terms in $F$ are unchanged. 

 We say $F^*$ is a 
\emph{complete reduction
of $F$ by $\cH$}, written $F\Rightarrow_{\cH}F^*$, if, for some $m$,  there is a sequence
$F_1=F,F_2,\dots, F_m=F^*$ such that for each $i=1,\dots, m-1$,  $F_i\to_{\cH}F_{i+1}$
is a simple reduction, and $F^*$ has no simple reduction.  Note that $F^*$ having no simple
reduction is equivalent to saying that all the paths $p$ in $F^*$ having nonzero coefficient in
$K[\mathbf  y]$ are in $ \cN$; that is, for all $t\in\cT$, $t$ is not a subpath of any $p$ occuring
in $F^*$.  Since $\succ$ is a well-order on $\cB$, every $F\in R$ will have a complete reduction.

{\bf We now  prove Theorem \ref{thm-var}.} 
Recall that 
$\cH=\{h_t\in R\mid  h_t=t-\sum_{n\in\cN_2(t)}y_{t,n}n\}$.
 For each pair  $t$ and $t'$ of elements of $\cT$ such
that $t=ab$ and $t'=bc$,
where $a,b$, and $c$ are arrows in $\cQ$, form the \emph{overlap relation}
\[ Ov(t,{t'})=h_t\cdot c-a\cdot h_{t'}.\]
Note that $t=t'=a^2$ is allowed. 
Since each $h_t$ is a $K[\mathbf y]$-combination of paths of length 2, each
overlap relation is a $K[\mathbf y]$-combination of paths of length 3.  We note that
if we have a $K[\mathbf y]$-combination of paths of length 3, then a simple reduction
is again a $K[\mathbf y]$-combination of paths of length 3.   It follows that
a complete reduction of a $K[\mathbf y]$-combination of paths of length 3 is again
a $K[\mathbf y]$-combination of paths of length 3.

Let $\cN_3$ denote the set of paths in $\cN$ of length 3.
For each $Ov(t,{t'})$, let
\[ F^*_{t,t'}=\sum_{\hat n\in\cN_3}f^*_{t,t',\hat n}(\mathbf y)\hat n,\] with
$f^*_{t,t',\hat n}(\mathbf y)\in K[\mathbf y]$, be a complete reduction of $Ov(t,{t'})$ by
$\cH$.  Thus, for each $t=ab$, $t'=bc$ and each $\hat n\in\cN_3$, we obtain polynomials in 
commutative polynomial ring $K[\mathbf y]$,  namely, the coefficient 
$f^*_{t,t',\hat n}(\mathbf y)$ of $\hat n$ in $F^*_{t,t'}$. 

We claim that $\GrAlg(\cT)$ is the zero set of \[
\cI=\{f^*_{t,t',\hat n}(\mathbf y)\mid \hat n\in\cN_3, t,t'\in \cT\text{ with }
t=ab\text{ and }t'=bc,\text{ for some arrows }a,b,c\}.\]

If, in the definitions of overlap relation, simple reduction, and complete reduction, instead of variables $y_{t,n}$ we use elements $x_{t,n}$ in $K$, we would have the definitions of overlap
relation, simple reduction, and complete reduction for elements of the
path algebra $K\cQ$.   The noncommutative version of Buchberger's Theorem \cite{Gr2,B},
applied to our setup, says
that if $\cG=\{g_t\mid t\in\cT\}$ is a uniform set of quadratic elements
in $K\cQ$ such that $\tip(g_t)=t$,
then $\cG$ is a \grb basis for $\langle \cG\rangle$ if and only if all overlap relations completely
reduce to $0$.   

 Suppose that  $X=(x_{t,n})\in\GrAlg(\cT)$.  We show that $X$ is in the zero set  of $\cI$.  We note that
$\cG(X)=\{g_t=t-\sum_{n\in\cN_2(t)}x_{t.n}n\mid t\in\cT \}$ is just $\cH$ evaluated at $X$.
Since $X\in \GrAlg(\cT)$, $\cG(X)$ is the reduced \grb basis for $\langle\cG(X)\rangle$ and
hence all
overlap relations of $\cG(X)$ reduce to 0.   Thus each $f^*_{t,t',\hat n}(X)=0$; in particular
$X$ is in the zero set of $\cI$.

Conversly, if  $X$ is in the zero set of $\cI $, each $f^*_{t,t',\hat n}(X)=0$.  Hence every
overlap relation of $\cG(X)$ completely reduces to 0, and we conclude that
$\cG(X)$ is a \grb basis of the ideal $\langle \cG(A)\rangle$.  Since $\tip(\cG(X))=\cT$
we see that $X\in \GrAlg(\cT)$.  This completes the proof.
\qed

\section{Properties of $\GrAlg(\cT)$}\label{sec-prop}
We begin with a general definition.

\begin{definition}\label{def-assoc}{\rm Given $\Lambda=K\cQ/I$, an arbitrary algebra , 
$K\cQ/\langle \tip(I)\rangle$ is called the \emph{associated monomial algebra} of $\Lambda$ and
denoted $\Lambda_{Mon}$. We also define $I_{Mon}$ to be $\langle\tip(I)\rangle$.}
\end{definition}

Note that given an ideal $I$, $\tip(I)$ is dependent on the choice of 
the admissible order $\succ$ and that, in this
paper, $\succ$ is fixed and has the property that, if $p,q\in\cB$ and the length of $p$ is greater
than the length of $q$, then $p\succ q$.

The next result provides an alternative definition of $\GrAlg(\cT)$.  Recall that if
$\mathbf x =
(x_{t,n})\in \cA=K^D$, then $\cG(\mathbf x)=\{g_t=t-\sum_{n\in\cN_2(t)}x_{t,n}n\mid
t\in\cT\}$.

\begin{proposition}\label{prop-mono} Let $\cT$ be a set of paths of length 2 in a quiver $\cQ$.
\sloppy
Let $\mathbf 0=(0,0,\dots,0)\in \cA$. The following statements hold:
\begin{enumerate}

\item $\cG(\mathbf 0)=\cT$.
\item The element $\mathbf 0\in\cA$ is in $\GrAlg(\cT)$ and corresponds to the strong Koszul
algebra $K\cQ/\langle\cT\rangle$ (with respect to $\langle\cT\rangle$ and $\succ$).
\item Let $\Lambda=K\cQ/I$ be a length graded algebra.  Then $\Lambda$ 
is a strong Koszul algebra (with respect to $I$ and $\succ$)
corresponding to a point in $\GrAlg(\cT)$  if and only if $I_{Mon}=\langle\cT\rangle$.
\item If $I$ is an ideal in $K\cQ$ generated by length homogeneous elements, then $K\cQ/I$ corresponds
to an element in the zero set of $\cI$ if and only if $(K\cQ/I)_{Mon}
=K\cQ/(I_{Mon})=K\cQ/\langle\cT\rangle$.
\item There is exactly one algebra with quadratic monomial \grb basis that corresponds to a point in $\GrAlg(\cT)$,
namely, $K\cQ/\langle\cT\rangle$. \end{enumerate}

\qed
\end{proposition}
The proof is straightforward and left to the reader.  As a consequence, we have the
following corollary.

\begin{corollary}\label{cor-mono}Let $\Lambda=K\cQ/I$ be a $K$-algebra with  length
grading induced from the length grading of $K\cQ$ .  The following  statements are equivalent:
\begin{enumerate}
\item
$\Lambda$
corresponds to an element of $\GrAlg(\cT)$.
\item $\Lambda_{Mon}=K\cQ/\langle\cT\rangle$.
\item $I_{Mon}=\langle\cT\rangle$.
\end{enumerate}

\end{corollary}

The next result shows that two algebras in a variety have the same $K$-bases of paths
under the spliting 
of the canonical surjection $\pi\colon K\cQ\to K\cQ/I$
given by  the Fundamental Lemma.  More precisely,
let $\sigma\colon K\cQ/I\to K\cQ$ be defined by $\sigma(\pi(x))=n_x$ where
$x=i_x+n_x$ with $i_x\in I$ and $n_x\in\Span_K(\nontip(I))$.  The map $\sigma$ is
well-defined 
 by the Fundamental
Lemma, and $\pi\sigma=1_{K\cQ/I}$. We identify
$\Lambda = K\cQ/I$ with $\Span_K(\nontip(I))$.  If $\Lambda\in\GrAlg(\cT)$, then
$\nontip(I)=\cB\setminus \tip(I)$.  Now $\langle\tip(I)\rangle=\langle \cT\rangle$ and hence $\nontip(I)=\{n\in\cB\mid n \text{ has
no subpath in }\cT\}$ by Proposition \ref{prop-tip-nontip}.  Thus, every algebra in $\GrAlg(\cT)$ has as $K$-basis
$\{n\in\cB\mid n \text{ has
no subpath in }\cT\}$. 
Of course, multiplication of elements of the basis differs for different algebras.

The converse holds.  More precisely, if $\Lambda=K\cQ/I$, where $I$ is generated
by uniform quadratic elements and $\nontip(I)=
\{n\in\cB\mid n \text{ has
no subpath in }\cT\}$, then $\Lambda$ is in $\GrAlg(\cT)$.  To see this, since
$\cB=\tip(I)\oplus \nontip(I)$, it follows that $\tip(I)=\{p\in\cB\mid \text{ there
is some }t\in\cT \text{ such that }t\text{ is a subpath of }p\}$.  From this
description it follows that $\langle \tip(I)\rangle=\langle \cT\rangle$, and hence
$\Lambda\in\GrAlg(\cT)$.

We summarize the above discussion in the next result, which
provides another description of $\GrAlg(\cT)$.

\begin{theorem}\label{thm-cartan} Let $\cT$ be a set of paths of length 2 in $\cQ$ and
$\cN=\cB\setminus \tip(\langle\cT\rangle)$.  We have
that if $I$ is generated by uniform length homogeneous elements, then $\Lambda=K\cQ/I\in\GrAlg(\cT)$ if and only if $\cN=\nontip(I)$.  Moreover,  if  $\Lambda=K\cQ/I\in\GrAlg(\cT)$ then, as a subspace of $K\cQ$,
$\Lambda$ has $K$-basis $\cN$.\qed
\end{theorem}

The following consequence of the previous  theorem describes the Cartan matrix of a strong
Koszul algebra and shows that two algebras in the same variety  have the same Cartan
matrix.

\begin{corollary}\label{cor-cartan}  Let $\cT$ and $\cN$ be as in Theorem \ref{thm-cartan}
and $\Lambda=K\cQ/I\in\GrAlg(\cT)$.  Suppose that $|\cN|<\infty$ and $\{v_1,\dots, v_n\}=\cQ_0$.
 Then the Cartan matrix of  $\Lambda$  is the $n\times n$ matrix $C$ where the $(i,j)$-th
entry in $C$ is $|v_i\cN v_j|$, the number of paths from $i$ to $j$ in $\cN$.
\end{corollary}
\begin{proof}   The
Cartan matrix is the $n\times n$ matrix with $(i,j)$-th entries $\dim_K(v_i\Lambda v_j)$.
But $\Span_K(\nontip(I))$ is isomorphic to $\Lambda$ and $\cN=\nontip(I)$

\end{proof}

The next result shows that algebras in the same variety share some homological properties.
Assume $I$ is an ideal in $K\cQ$ contained in $J^2$, where $J$ is the ideal in $K\cQ$ generated
by the arrows of $\cQ$.   If $v$ is a vertex in $\cQ$ and
$\Lambda=K\cQ/I$, we let $S_v(\Lambda)$ be the one-dimensional simple $\Lambda$-module
associated to the vertex $v$.

Note that if $\Lambda=K\cQ/I\in\GrAlg(\cT)$, $I\subseteq J^2$ since $I$ has a \grb bases consisting
of quadratic elements.  

If $\Lambda$ is a ring and $M$ is a $\Lambda$-module, we let $\pd_{\Lambda}(M)$
and $\id_{\Lambda}(M)$
denote the projective and injective dimensions of $M$ respectively.

\begin{theorem}\label{thm-ext}  Let $\cT$ be a set of    paths of length 2 in $\cQ$ and $\Lambda\in\GrAlg(\cT)$.
Suppose that $v$and $w$ are vertices in $\cQ$. Then for $n\ge 0$,
\[ \dim_K(\Ext_{\Lambda}^n(S_v(\Lambda),S_w(\Lambda))=
 \dim_K(\Ext_{\Lambda^*}^n(S_v(\Lambda^*),S_w(\Lambda^*)),
\]
where $\Lambda^*=K\cQ/\langle\cT\rangle$.  In particular, if $\cN=\cB\setminus \tip(\langle\cT\rangle)$, then
\[\pd_{\Lambda}(S_v(\Lambda))=\pd_{\Lambda^*}(S_v(\Lambda^*))
\text{ and }\id_{\Lambda}(S_v(\Lambda))=\id_{\Lambda^*}(S_v(\Lambda^*)).\]
.
\end{theorem}

\begin{proof}  Although the proof of this result is implicit in \cite{AG}, we 
sketch a proof employing the ideas in \cite{GS}.  In \cite{GS}, a projective $\Lambda$-resolution
is constructed inductively  from subsets $F^n=\{f^n_i\}_{i=1}^{k_i}$ for $n\ge 0$, where
 $f^n_i\in K\cQ$ and
  $k_i \text{ is finite}$ if
the reduced \grb basis is finite, which it is in our case.   
For $S_v(\Lambda)$, we have $F^0=\{v\}$, $F^1$ is the set of arrows starting at $v$, and
$F^2$ are the elements of a reduced \grb basis that start at $v$.   The $F^n$'s are constructed
using the $F^i$'s, $i<n$ and the reduced \grb basis. 

 We describe $\tip(F^n)$, for $n\ge 2$ which can be deduced from the construction of $F^n$
from $F^{n-1}$.
The construction shows that  $\tip(F^n)=T^n$, where
\[T^n=\{a_1a_2\cdots a_n\mid a_i\in\cQ_1, va_1=a_1, \text{ and }a_ia_{i+1}\in\cT,\text{ for }1\le i\le n-1\}.\]
Note that $T^n$ depends only on $\cT$,
and that  $|T^n|=|F^n|$.   Since the $f^n_i$ are  length homogeneous,
we see that if $f^n_i\in F^n$ and $\tip(f^n_i)=a_1\cdots a_n\in T^n$,  then
$f^n_i$ is length homogeneous of length $n$.  Since each $f^n_i$ is uniform, if $w$ is the end vertex of $a_n$, then $vf^{n}_iw=f^{n}_i$.

From the construction of the
$\{f^n_i\}$, 
each $f^n_i$ is a sum of elements of the form $f^{n-1}_jr_{i,j}$ with $r_{i,j}\in K\cQ$.   By length,
the $r_{i,j}$ are linear combinations of arrows.  Since the $r_{i,j}$ modulo $I$ are entries in the matrix
mapping the $n^{th}$ projective to the $n-1^{st}$ in the constructed projective resolution
of $S_v(\Lambda)$, we conclude that the resolution constructed in \cite{GS} is minimal in
our case.

Finally, we see that if 
$F^n$ is the set for resolving $S_v(\Lambda)$,
and ${F^*}^n$ is the set for resolving $S_v(\Lambda^*)$, then they both have tip
set $T^n$.   This finishes the proof since
the dimension of $\Ext_{\Lambda}^n(S_v(\Lambda),S_w(\Lambda))$ equals the
number of $f^n_i$s such that $vf^n_iw=f^n_i$.  
\end{proof}

We have the following consequence.

\begin{corollary}\label{cor-gldim}If $\cN=\cB\setminus \tip(\langle \cT\rangle)$ is a finite
set and
$\Lambda,\Lambda'\in\GrAlg(\cT)$, then 
\[\gldim(\Lambda)=\gldim(\Lambda').\]

\end{corollary}
\begin{proof} Let $\Lambda,\Lambda'\in\GrAlg(\cT)$.  Since $|\cN|=\dim_K(\Lambda)=
\dim_K(\Lambda')$, $\gldim(\Lambda)=\max_{v\in\cQ_0}\{\pd_{\Lambda}(S_v(\Lambda))$
and $\gldim(\Lambda')=\max_{v\in\cQ_0}\{\pd_{\Lambda'}(S_v(\Lambda'))\}$.  The result
follows from Theorem \ref{thm-ext}.

\end{proof}

If $\Lambda^*=K\cQ/\langle \cT\rangle$ is a finite
dimensional strong Koszul algebra (with respect to $\langle\cT\rangle$ and
$\succ$) and of finite global dimension, then the determinant of the Cartan matrix for 
every algebra in $\GrAlg(\cT)$ is $1$,  since every algebra in $\GrAlg(\cT)$ is  length graded and of finite global 
dimension and, hence,  we may apply \cite{W}.

The final property is one that is proved in a more general setting in \cite{GHS}.

\begin{theorem}\label{thm-qh}\cite{GHS} Let $\cT$ be set of paths of length 2 in $\cQ$ and
$\cN=\cB\setminus \tip(\langle \cT\rangle)$.  Assume that $\cN$ is a finite set.  If
$K\cQ/\langle\cT\rangle$ is a quasi-heredity algebra, then every algebra in $\GrAlg(\cT)$ is
quasi-heredity.
\end{theorem}

\section{Examples}\label{sec-ex}

We begin by defining a particular length admissible order.  Unlike the commutative case, the (left)
lexicographic order is not a well-order on $\cB$.  In particular, the lexicographic order
will have infinite descending chains, contradicting that it must be a well-order.

We describe the \emph{length-left-lexicograpic order} which is length admissible. Arbitrarily
linearly order the vertices and arrows of $\cQ$  and require that every 
vertex is less than every arrow.  If $p=a_1a_2\cdots a_n$ and $q=b_1b_2\cdots b_m$
are paths with the $a_i$ and $b_j$  arrows, then $p\succ q$ if $n>m$ or
$n=m$ and there is $k$, $1\le k\le n$ such that for $1\le i\le k-1$, $a_i=b_i$ and
$a_k\succ b_k$.

In all the examples in this section we will use the length-left-lexicographic order and it will suffice,
in our setting, to simply
give the ordering of the arrows.   Recall that if $\cT$ is a set of paths of length 2 in
a quiver $\cQ$ and $\cG=\{g_t\mid t\in\cT\}$ where, for $t\in\cT$,
$g_t=t-\sum_{n\in\cN_2(t)}x_{t,n}n$ for $n\in\cN_2(t)$,  $x_{t,n}\in K$, then
$\cG$ is the reduced \grb basis for the ideal $\langle\cG\rangle$ if and only
if all overlap relations completely reduce to 0 by $\cG$. 

There are two extreme cases given $\cT$ and $\cG$ as above.  The first occurs
if there are no overlap relations.  In this case, the ideal of the variety $\GrAlg{\cT}$
is $(0)$ and hence $\GrAlg{\cT}$ is all of affine space and there are no
restrictions on the choice of the $x_{t,n}$.   The second case extreme case
occurs if $\cN_2(t)=\emptyset$, for all $t\in\cT$.  For example, this occurs
if for each $t\in\cT$, if $t\succ t'$ and $t'$ is a path of length 2, then
$t'\in \cT$ . In this case, the affine space $\cA$ is dimension 0 and 
 $\GrAlg(\cT)=\cA$ is a point, namely the monomial algebra
$K\cQ/\langle\cT\rangle$.

We now turn to specific examples.  The next example is designed to help understand
the proof of Theorem \ref{thm-var}. 

\begin{Example}\label{ex-gd}{\rm  Let $Q$ be the quiver
\xymatrix{
&&\circ \ar[ld]_a\ar[d]^b\ar[rd]^c\\
&\circ\ar[ldd]_e \ar[ddrrr]_f  &\circ \ar[ddll]_g\ar[drrd]^h        &\circ\ar[ddlll]^i\ar[ddr]^j\\\\
\circ\ar[drr]^k&&&&\circ\ar[dll]_l\\
&&\circ
} and
let $\succ$ be\linebreak  defined by $a\succ b\succ \cdots \succ l$.  Consider
$\cT=\{ af,ae,bg,bh, ek,gk,ik\}$.
Then, it follows that $\cN= \cQ_0\cup \{a,b,c,\dots,k,l,ci, cj, fl,hl,jl,cfk,cjl\}$. 
Thus we have $\cN_2(af)=\{cj\},   \cN_2(ae)=\{ci\},\cN_2(bg)=\{ci\},
\cN_2(bh)=\{cj\},\cN_2(ek) =\{fl \}, \cN_2(gk) = \{hl\},$ $ \cN_2(ik) =\{jl\}$.

Thus $\GrAlg(\cT)$ is a variety in $\cA=K^{7}$.
Simplify notation by renaming the variables $y_{t,n}$ where $t\in\cT$ and
$n\in \cN_2(t)$ as follows:
$ y_{af,cj}=X_1,\ y_{ae, cf}=X_2,\ y_{bg,ci}=X_3,\
 y_{bh,cj}=X_4,\ y_{ek,fl}=X_5,\ y_{gk,hl}=X_6,\ y_{ik,jl}=X_7$. 
We let
\[\cG=\{af-X_1cj,\ ae-X_2cf,\ bg-X_3ci, \dots, ik-X_7jl\}.\]

There are 2 overlap relations
$Ov(ae,ek)= -X_2cik+X_5afl$  and
$Ov(bg,gk)= -X_3cik +X_6bjl$.  We completely reduce the first
overlap relation and leave the computation of the second to the
reader.  We see that  $-X_2cik$  
simply reduces to $-X_7X_2cjl$ by $ik-X_7jl$.  Since $cjl\in\cN$ it 
has no simple reductions.  Next consider the second term
$X_5afl$.   Using $af-X_1cj$, $X_5afl$ simply reduces to $X_1X_5cjl$ and, as we noted,
$cjl\in\cN$.  Thus $Ov(ae,ek)$ completely reduces to
$-X_2X_7cjl + X_1X_5cjl$.   Similarly,
$Ov(bg,gk)\Rightarrow_{\cG} -X_3X_7cjl  +    X_4X_6cjl$.

Thus, the ideal of the variety is $\cI=\langle X_1X_5-X_2X_7,\ X_4X_6-X_3X_7\rangle$
in the commutative polynomial ring $K[X_1,X_2,\dots, X_7]$.  Note that using \cite{GHZ}
we see that the $\gldim(K\cQ/\langle \cT\rangle = 3$.  Thus by Corollary \ref{cor-gldim} and
Theorem \ref{thm-cartan},
every algebra in $\GrAlg(\cT)$ is a strong Koszul algebra of dimension 26 with $K$-basis
 $\cN$, and has global dimension 3.
\qed

}\end{Example}

We denote the free associative algebra in $n$ variables by $K\{x_1,\dots,x_n\}$.
Our next example is  very small and simple.  In this  example,
$\GrAlg(\cT)$ is affine 2-space and there is a punctured line in $\GrAlg(\cT)$ consisting
of  the
quantum affine planes.  Moreover there is another line in $\GrAlg(\cT)$ on which all
the algebras are isomorphic to the monomial algebra $K\{x, y\}/\langle \cT\rangle$.
This provides an example of distinct points in $\GrAlg(\cT)$ corresponding to
isomorphic algebras.

\begin{Example}\label{ex-poly2}{\rm Let $R =K\{x,y\}/\langle xy-yx\rangle$ and
$y\succ x$.  In this example, $\cQ$ has one vertex and two loops.
Then $\cG=\{yx-xy\}$ and $\cT=\{yx\}$.  The paths of length 2 are 
ordered $y^2\succ yx\succ xy\succ x^2$.
  We see that $\cN=\{x^iy^j\mid i,j\ge 0\}$.
$\cN_2(yx)=\{xy,x^2\}$.  Thus, $\GrAlg(\{yx\})$ lives in $\cA=K^2$.  There are no
overlap relations and hence $\GrAlg(\{yx\})=\cA$.  
Thus, every algebra of the form $\Lambda_{(\lambda,\gamma)}=
K\{x,y\}/\langle yx-\lambda xy-\gamma x^2\rangle$ is in $\GrAlg(\{yx\})$,
where $(\lambda,\gamma)\in K^2$.
Note that in the notation of Section \ref{sec-var},
$\lambda=y_{yx,xy}$ and $\gamma =y_{yx,x^2}$.

If $\gamma=0$ and $\lambda\ne 0,1$, then $\Lambda_{(\lambda,0)}$ is a quantum
affine plane.
They all lie on the punctured line 
$L=\{(\lambda,0)\mid \lambda\in K\setminus \{0,1\}\}$ in $\cA$.  We also have $\Lambda_{(1,0)}=R$, the commutative polynomial ring
in 2 variables.  Of course, $\Lambda_{(0,0)}$ is the (noncommutative) monomial algebra $K\{x,y\}/\langle yx\rangle$.  

On the other hand, the line determined by $\lambda=0$ consists of the algebras
$K\{x,y\}/\langle yx-\gamma x^2$.   It is not hard to show these algebras are
all isomorphic to each other.  In particular, they are all isomorphic to the monomial
algebra $K\{x,y\}/\langle yx\rangle$.

If both $\lambda$ and $\gamma$ are not 0, then other strong Koszul algebras occur.

Finally, for every $\Lambda\in\GrAlg(\{yx\})$  the minimal projective $\Lambda$-resolution
of $K$ looks like
\[0\to \Lambda\to \Lambda^2\to  \Lambda \to K\to 0.\]
\qed

}\end{Example} 

In the next small example, the algebras occuring are finite dimensional.

\begin{Example}{\rm
Again take $\cQ$ having one vertex and two loops, $x$ and $y$.  Again
we order $ y\succ x$.  Let $\cT=\{x^2,y^2, yx\}$.  Then $\cN=\{1,x,y, xy\}$
Hence $\cN_2(x^2)=\emptyset,\  \cN_2(y^2)=\{xy\}=\cN_2(yx)$.  Thus $\cA$ is
two space.  Now let $U$ and $V$ be variables.  We wish to find the ideal $\cI$ of
the variety $\GrAlg(\cT)$.  We have $\cG=\{x^2,y^2-Uxy,yx-Vxy\}$. Note that in
the notation of Section \ref{sec-var}, $U=y_{y^2,xy}$ and $V=
y_{yx,xy}$.  To find the polynomials
in $\cI$, we need to completely reduce all overlap relations.

There are 3 overlap relations
 \[Ov(x^ 2,x^2)=0, Ov(y^2,y^2)=Vxy^2-Vyxy, \text{ and } Ov(y^2,yx)=-Vxyx+Uyxy.\]
But there are no length 3 paths in $\cN$, hence the three overlap relations must
completly reduce to 0 by $\cG$. It follows that $\cI=\{0\}$.
Thus, $\cA=\GrAlg(\cT)$.  Note that for $V=1$ we have commutative strong Koszul
algebras, including $K\{x,y\}/\langle x^2,y^2, yx-xy\rangle$.
 \qed

}\end{Example} 

These examples are deceptively easy.  In general, $\GrAlg(\cT)$ is a proper nontrivial
variety.   The next example gives some indication of the complexity of the varieties.

\begin{Example}\label{ex-poly3}
{\rm Let $\cQ$ be the quiver with one vertex and 3 loops, $x,y$, and $z$.
Order them by $z\succ y\succ x$.  Let $\cT=\{ zy,zx,yx\}$.  Then $\cN=\{x^iy^jz^k\mid i,j,k\ge 0\}$.
We note that the commutative polynomial ring $K\{x,y,z\}/\langle zy-yz, zx-xz, yx-xy\rangle$ is a
strong Koszul algebra since $Ov(zy,yx)=-yzx+zyx$ completely reduces to 0 by $\cG=\{ zy-yz, zx-xz, yx-xy\}$.

By the results in Section \ref{sec-prop}, the strong Koszul algebras having basis $\cN$ (the same
as the commutative polynomial ring in 3 variables ) are the points in $\GrAlg(\cT)$.  The projective
resolution of $K$ over these algebras all have the same shape as the resolution of $K$ over the
commutative polynomial ring by Theorem \ref{thm-ext}.  Thus $\GrAlg(\cT)$ has some connection with Artin-Schelter regular
algebras of dimension 3.

Now $\cN_2=\{z^2,yz,y^2,xz,xy,   x^2\}$  and the length-lexicographic order yields
\[z^2\succ zy\succ zx\succ yz\succ y^2 \succ yx\succ 
xz\succ xy\succ x^2\]
Thus, $\cN_2(zy)=\{yz,y^2,xz,xy,x^2\}=\cN_2(zx)$ and
$\cN_2(yx)=\{xz,xy,x^2\}$.   Hence $\cA=K^{13}$ and $\GrAlg(\cT)$ is a variety
in $\cA$.

We let $A,B,\dots,H,L,M,N,P,Q$ denote 13 variables.
We set \begin{enumerate}
\item[] $g_{zy}=zy-Ayz-By^2-Cxz-Dxy-Ex^2$
\item[] $g_{zx}=zx-Fyz-Gy^2-Hxz-Kxy-Mx^2$
\item[] $g_{yx}=yx-Nxz-Pxy-Qx^2$.
\end{enumerate}
Let $\cG=\{g_{zy},g_{zx},g_{yx}\}$.

There is only one overlap relation we need to study, namely,
\[
Ov(zy,yx)=-Ayzx-By^2x-Cxzx-Dxyx-Ex^2+Nzxz+Pzxy+Qzx^2.\]
We completely reduce this overlap relation and collect the
polynomials that are the coefficients of elements of $\cN_3$.
By brute force computation, the ideal of $\GrAlg(\cT)$ contains
8 polynomials, 2 of total degree 3 and 6 of total degree 4.
 
Below is the polynomial that is the coefficient of $xyz$:
\[\begin{array}{l}
-AP-A^2LN -AFMN-BP-BPNA-BQNF-CF+NFP+N^2DA+N^EF\\
+NHA+P^2F+PDNA+PENF +PHA+QFP+BLNA+QMNF\\
+MGNP+MPNA+MQNF+MHF
\end{array}\]  The dimension of $\GrAlg(\cT)$ is unclear,    as is its irreducibility.
\qed

}
\end{Example}

\begin{remark}{\rm

The connection with Artin-Shelter regular algebras holds in all dimensions.  More precisely, consider
the commutative polynomail in $n$ variables $R=K\{x_1, x_2,\dots, x_n\}/
\langle x_jx_i-x_ix_j\mid 1\le i<j\le n\rangle $.  Taking  $x_n\succ x_{n-1}\succ\cdots\succ x_1$,
it  is easy to check that $R$ is a strong Koszul algebra.  Thus, if $\cT=\{x_jx_i\mid 1\le i<j\le n\}$
then $R\in\GrAlg(\cT)$ and, in this case ,
$\cN=\{x_1^{i_1}x_2^{i_2}\cdots x_n^{i_n} \mid i_j\ge 0, !\le i\le n\}$.
Thus $\GrAlg(\cT)$ consists of all the strong Koszul algebras with graded $K$-bases $\cN$.
Again, for each algebra in $\GrAlg(\cT)$, the simple module $K$ has a projective resolution
the same shape as the projective resolution of $K$ over $ R$ by Theorem \ref{thm-ext}. 
}\end{remark}

\section{Subvarieties}\label{sec-sub}
Fix  a quiver $\cQ$, a length admissible order $\succ$ on the set of paths $\cB$, and a
set $\cT$ of paths of length 2.    As usual, let $\cN=\cB\setminus \tip(\langle\cT
\rangle)$ and $D=\sum_{t\in \cT}\mid\cN_2(t)\mid$.   As we have
seen, $\GrAlg(\cT)$ is a variety in $\cA=K^D$ 
and the points of $\GrAlg(\cT)$ correspond to the strong Koszul algebras with a
fixed associated monomial algebra $\K\cQ/\langle \cT\rangle$.   In this section
we introduce subvarieties of $\GrAlg(\cT)$ that have a distinguished subalgebra
and which are  intersections of
$\GrAlg(\cT)$ with  specified affine subspaces.  

Let $\mathfrak r$ be a subset of
$\{(t,n)\mid t\in \cT \text{ and }n\in \cN_2(t)\}$ and $\psi\colon \mathfrak r\to K$.
We define $\GrAlg_{\psi}(\cT)$ to be the set of the strong Koszul algebras
$\Lambda=K\cQ/I$ (with respect to $I$ and $\succ$) such that the reduced \grb
basis of $I$,
$\cG=\{g_t\mid t\in\cT\}$, where  $g_t =t-\sum_{n\in\cN_2(t)}c_{t,n} n$ and satisfies the
restriction
that, for each $(t,n)\in\mathfrak r, c_{t,n}=\psi((t,n))$.

We break $\mathfrak r$ into two disjoint sets; namely,
Let $\mathfrak r^0=\{(t,n)\in\mathfrak r\mid \psi((t,n))= 0\}$ and $\mathfrak r^+=
\mathfrak r\setminus \mathfrak r^0$.
The \emph{distinguished
algebra} in $\GrAlg_{\psi}(\cT)$ is $\Lambda^*=K\cQ/I^*$ where $I^*$ is generated
by $\{g^*_t\mid t\in\cT \}$  where $g^*_t =t-\sum_{(t,n)\in\mathfrak r^+} \psi((t,n))n$.
Note that if $\mathfrak r^+=\emptyset$, then $g^*_t=t$ and $\Lambda^*=K\cQ/
\langle\cT\rangle$.    In general, $\Lambda^*$ may or may not be in
$\GrAlg_{\psi}(\cT)$, depending on whether or not $\Lambda^*$ is a strong
Koszul algebra.  Summarizing,
 $\GrAlg_{\psi}(\cT)$ are the strong Koszul algebras in $\GrAlg(\cT)$ having
reduced \grb bases  $\{t-\sum_{n\in\cN_2(\cT)}c_{t,n}n\}$ with $c_{t,n}=\psi((t,n))$
for all $(t,n)\in\mathfrak r\}$.   

    Let $\mathfrak A_{\psi}$ be the affine subspace of $\cA$
defined by \[\mathfrak A_{\psi} =\{(x_{t,n})\in \cA\mid x_{t,n}=\psi((t,n))\text{ if } (t,n)\in\mathfrak r\}\].    We have the following result, whose proof is left to the reader..

\begin{proposition}\label{prop-inter}Let $\cT$ be a set of paths of length 2 in
a quiver $\cQ$.  If $\mathfrak r\subseteq  \{(t,n)\mid t\in\cT, n\in\cN_2(t)\}$
and $\psi\colon \mathfrak r\to K$, then
\[ GrAlg_{\psi}(\cT)=\GrAlg(\cT)\cap \mathfrak A_{\psi}.\] Moreover,
$\dim_K(\mathfrak A_{\psi})=D-|\mathfrak r|$. \qed
\end{proposition}

By the above Proposition,  $\GrAlg_{\psi}(\cT)$ could be viewed
as living in $K^{D- |\mathfrak r|}$.   In this setting 
the distinguished algebra is  the algebra associated to the point $\mathbf 0\in
 K^{D- |\mathfrak r|}$.  The next result provides another proof that
    $\GrAlg_{\psi}(\cT)$ is an affine variety in affine $(D-|\mathfrak r|)$-space
by explicitly describing the ideal of the variety, viewed as a variety in
$K^{|D|-|\mathfrak r|}$.

\begin{theorem}\label{thm-sub}
Keeping the notation above, $\GrAlg_{\psi}(\cT)$ is a subvariety   of $\GrAlg(\cT)$.
The subvariety   $\GrAlg_{\psi}(\cT)$ lives in affine space  $K^{D'}$, where $D'=(\sum_{t\in\cT} |\cN_2(t)|) -|\mathfrak r|$.
\end{theorem}

\begin{proof}  The proof follows the proof of Theorem \ref{thm-var} after replacing
the variables $y_{t,n}$ with the constants $\psi((t,n))$ for $(t,n)\in \mathfrak r$. Thus
the polynomials in the ideal of the variety are produced by completely reducing the appropriate
overlap relations.  (See Example \ref{ex-poly3-1} below.)
The dimension of the underlying affine space is clear.
\end{proof}

\begin{Example}\label{ex-poly3-1}{\rm  Suppose we are interested in strong Koszul algebras
in the same variety as the commutative polynomial ring $R=K\{x,y,z\}/\langle
zy-yz,zx-xz,yx-xy\rangle$ with $z\succ y\succ x$.  We saw in Example \ref{ex-poly3} that
$\GrAlg(\cT)$ where $\cT=\{zy,zx,yx\}$ is extremely complicated.  We consider
the following subvariety. 

Let \sloppy
$\mathfrak r=\{(zy,yz),(zy,y^2),(zy,xz),(zy,x^2),(zx,yz),(zx,y^2),(zx,xz),(zx,x^2),(yx,xz),$
\linebreak $(yx,xy)\}$ and set
$\psi((zy,yz))=\psi((zx,xz))=\psi((yx,xy))=1$ with all other values of $\psi$ being 0.
Note that the distinguished algebra in $\GrAlg_{\psi}(\cT)$ is $R$, the commutative
polynomial ring in 3 variables..
There three $(t,n)$s not in $\mathfrak r$: $(zy,xy), (zx,xy), (yx,x^2)$.
Let $A=c_{zy,xy}, B=c_{zx,xy}$ and $C=c_{yx,x^2}$.  We have
\[\cG=\{ g_{zy}=zy-yz -Axy, g_{zx}=zx-xz+Bxy, g_{yx}=yx-xy+Cx^2\}.\]
It follows that $\GrAlg_{\psi}(\cT)$ is the variety whose points $(A,B,C)$ 
satisfy the property that $\cG$ is the reduced \grb for the ideal it generates
in $K\{x,y,z\}$.

To find ideal of the variety, we must completely reduce  every overlap relation .  As we
noted in Example \ref{ex-poly3}, there is only one overlap relation, namely,
\[ Ov(zy,yx)=-yzx-Axyx +zxy +Czx^2.\]
After completely reducing the overlap relation, the polynomials that are coefficients of $\cN_3$
generate $\cI$.  The reader can verify that we obtain two polynomials: $BC^2-AC$ and $BC$.
Thus, $\cI=\langle BC^2-AC, BC\rangle =\langle AC, BC\rangle$.
The variety $\GrAlg_{\mathfrak r}(\cT)$ in $K^3$ has two irreducible components
the plane $\{(A,B,0)\}$ and the line $\{(0,0,C)\}$. the commutative polynomial
ring lies in the plane $\{(A,B,0)\}$.

}\end{Example}

\section{Further results  on strong Koszul algebras}\label{sec-results}

We begin by looking at $\Lambda^{op}$, the opposite algebra of $\Lambda$.
The opposite algebra of $\Lambda$ is $\{\lambda^{op}\mid \lambda\in \Lambda\}$ with
addition and multiplication given by $\lambda^{op}+(\lambda')^{op}=(\lambda+\lambda')^{op}$
and $\lambda^{op}\cdot (\lambda')^{op}=(\lambda'\cdot\lambda)^{op}$.  It is well-known
that $\Lambda$ is a Koszul algebra if and only if $\Lambda^{op}$ is a Koszul algebra.
If $\succ$ is a length admissible order, then let $\succ^{op}$ be the order
$p^{op}\succ^{op}q^{op}$ if and only if $p\succ q$.  In particular, if $\succ$ is the
length-(left)-lexicographic order defined in Section \ref{sec-ex}, then $\succ^{op}$
is the length -(right)-lexicographic order. 

Given a quiver $\cQ$, define the opposite quiver, $\cQ^{op}$ in the obvious
way.  If $I$ is an ideal in
$K\cQ$, then let $I^{op}=\{x^{op}\mid x\in I\}$.    We have the following  result, whose proof
is left to the reader.

\begin{proposition}\label{prop-op} The algebra $\Lambda=K\cQ/I$ is a strong Koszul
algebra (with respect to $I$ and $\succ$) if and only if $\Lambda^{op}=K\cQ^{op}/I^{op}$
is a strong Koszul algebra (with respect to $I^{op}$ and $\succ^{op}$). \qed
\end{proposition}

The next result deals with tensoring two strong Koszul algebras.  In fact, we 
prove a  result about the reduced \grb basis of the tensor of two algebras in general; see
Theorem \ref{thm-tensor}(1) below. Let $\Lambda=K\cQ/I$
and $\Lambda'=K\cQ'/I'$.  Define $Q^*$ to be the quiver with vertex set
$\cQ_0\times \cQ_0'$ and arrow set 
$
(\cQ_1\times\cQ_0')\cup (\cQ_0\times \cQ'_1)
$,
where $(a,w')\colon (u,w')\to (v,w')$ if $a\colon u\to v$ and
$(v,b')\colon (v,w')\to (v,x')$ if $b'\colon w'\to x'$.  If $p=a_1a_2\cdots a_r$  and
$w'\in\cQ'_0$, then let $(p,w')$ denote the path $(a_1,w')(a_2,w')\cdots (a_r,w')$.
If $q'$ is a path in $\cQ'$ and $v\in\cQ_0$, then $(v,q')$ has a similar meaning.
If $r=\sum_{p\in\cB}\alpha_pp\in K\cQ$ and $w'\in\cQ'_0$, then let
 $(r,w)=\sum_{p\in\cB} \alpha_p(p,w')$.  Similarly, $(v,\sum_{q'}\beta_{q'}q')=
\sum_{q'}\beta_{q'}(v,q')$.

Define $\varphi\colon K\cQ^*\to \Lambda\otimes_K\Lambda'$ as follows. If $(v,w')\in\cQ^*_0$, $\varphi(v,w')=v\otimes w'$, if $(a,w') \in\cQ_1\times \cQ'_0$, $\varphi(a,w')=a\otimes w'$,  and
if $(v,b') \in\cQ_0\times \cQ'_1$, $\varphi(v,b')=\otimes b'$.  This ring homomorphism
is clearly surjective.  Let $I^*$ denote the kernel of this morphism.  
Finally, let $\succ$ and $\succ'$ be length admissible orders on $\cB$, the set of paths in $\cQ$, 
and
on $\cB'$, the set of paths in $\cQ'$, respectively. 

 Let $\succ^*$ be the length admissible order
on $\cB^*$, the set of paths in $\cQ^*$, be defined as follows: On vertices $(u,w')\succ^*
(v,x')$ if $w'\succ' x'$ or $w'=x'$ and $u\succ v$. 
Let $p^*$ be the path $(v_1,q'_1)(p_1, w_1')(v_2,q'_2)\cdots (p_n,w_n')$
and $\hat p^*$ be the path  $(\hat v_1,\hat q'_1)(\hat p_1,\hat w_1')(\hat v_2,\hat q'_2)\cdots (\hat p_m,\hat w_m')$ with $p_i,\hat p_\in\cB, q'_j,\hat q'_j\in \cB'$.  Remove
all the $\cQ^*$ vertices from $p^*$ and $\hat p^*$.  Then
$p^*\succ^* \hat p^*$ if $\ell(p^*)>\ell(\hat p^*)$ or $\ell(p^*)=\ell(\hat p^*)$
and the first arrows from the left in $p^*$ and $\hat p^*$ where the arrows differ, we have one of the following
4 cases: 
\begin{enumerate}
\item the arrow in $p$ is  $(v,b')$ and the arrow in $\hat p$ is $(u,b'')$
with $b',b''\in\cQ_1'$ and $b'\succ' b''$.
\item the arrow in $p$ is  $(v,b')$ and the arrow in $\hat p$ is $(a,w')$
with $b'\in \cQ'_1$ and $a\in \cQ_1$.
\item the arrow in $p$ is  $(a,w')$ and the arrow in $\hat p$ is $(a',w'')$ with
$a,a'\in\cQ_1$ and $a\succ a'$.
\item the arrow in $p$ is $(a,w')$ and the arrow in $\hat p$ is $(a,w'')$ with 
$a\in\cQ_1$ and $w'\succ' w''$.
\end{enumerate}
The reader may check that $\succ^*$ is a length admissible order.

Before getting to the result on tensors,  we let \[C=\{(u,b')(a,x')-(a,w')(v,b')\mid
a\colon u\to v\text{ is an arrow in }\cQ_1, \]\[ b'\colon w'\to x'\text{ is an arrow in }\cQ'_1\}.\]  We see that $C\subseteq I^*$.  Let $\cC$ be ideal
in $K\cC$ generated by $C$.  The elements of $C$ are called \emph{commutativity relations}.
Note that they are quadratic elements in $K\cQ^*$.  The proof of the following result
is left to the reader.

\begin{lemma}\label{lem-comm} Let $p^*$ be a path in $\cB^*$ of length at least 1.
Then there exist paths  $p\in \cB$, $q'\in\cB'$ such that
either $\ell(p)\ge 1$ or $\ell(q')\ge 1$ and
$p^*-(p,w')(v,q')\in \cC$, where $v\in\cQ_0$ is the end vertex of $p$ and
$w'\in\cQ'_0$ is the start vertex of $q'$ .  \qed
\end{lemma}

We use the following convention: If $z$ is a uniform element in $K\cQ$ with $uzv=z$, $u,v\in\cQ_0$
 and $y'$ is a uniform element in $K\cQ'$ with$w'y'x'=y'$, then
we define $(z,y')\in K\cQ^*$ to be $( z,w')(v,y')$. Note that
$(z,w')(v,y')-(u,y')(z,x')$ is an element of $\cC$.
We also have the following lemma.

\begin{lemma}\label{lem-basis} Suppose that $\Lambda=K\cQ/I$, $\Lambda'=K\cQ'/I'$
are $K$-algebras, and $K\cQ^*$ and $I^*$ are  defined above.  Let $\succ$ and
$\succ'$ be length admissible orders for $\cB$ and $\cB'$, respectively.  Set
$\cN=\cB\setminus\tip(I)$ and $\cN'=\cB'\setminus\tip(I')$.  Then the
set $\{n\otimes_Kn'\mid n\in\cN, n'\in\cN'\}$ is a $K$-basis for
$\Lambda\otimes_K\Lambda'$.   
\end{lemma}

\begin{proof}  
By the Fundamental Lemma, we have that $\cN$ is a $K$-basis of $\Lambda$ and
$\cN'$ is a $K$-basis of $\Lambda'$.  Since we are tensoring over $K$, the result
follows.

\end{proof}
The first part of the
following result is quite general.  The proof of the result 
is somewhat technical but straightforward, and we leave routine checking to the reader. 

\begin{theorem}\label{thm-tensor} Suppose that $\Lambda=K\cQ/I$, $\Lambda'=K\cQ'/I'$
are $K$-algebras, and $K\cQ^*$ and $I^*$ are  defined above.  Let $\succ$ and
$\succ'$ be length admissible orders for $\cB$ and $\cB'$, respectively, and let
$\succ^*$ be the length admissible order defined above.  Let $\cG$ and $\cG'$ be
the reduced \grb  bases for $I$ and $I'$, respectively.  Then
\begin{enumerate}
\item  \[\cG^* =\{(g,w')\mid g\in\cG, w'\in\cQ'_0\}\cup
\{(v,g')\mid g'\in\cG', v\in\cQ_0\}\cup C\] is the reduced  \grb basis for $I^*$ with respect to
$\succ^*$.  If $\cG$ and $\cG'$ are composed of length homogeneous elements, then
so is $\cG^*$ and in this case $\cG^*$ is a graded \grb basis for the induced length graded
algebra $K\cG^*/I^*=\Lambda\otimes_K\Lambda'$.
\item If $\Lambda=K\cQ/I$ and $\Lambda'=K\cQ'/I'$ are strong Koszul algebras  (with
respect to $I$ and $\succ$) and (with respect to $I'$ and $\succ'$), respectively, then
$\Lambda\otimes_K\Lambda'=K\cQ^*/I^*$ is a strong Koszul algebra (with respect to
$I^*$ and $\succ^*$), 
\end{enumerate} 
\end{theorem}

\begin{proof}
First  we  show that $\cG^*$ generates $I^*$.
Clearly $\cG^*\subseteq I^*$.  
Suppose $X=\sum_i\alpha_{p_i^*}p_i^*\in I^*$, where $\alpha_{p_i^*}\in K$ and
$p_i^*\in\cB^*$.  Then, by Lemma \ref{lem-comm}, since $\cC\subset I^*$, by
repeatedly applying the commutativity relations, we
may assume that $X=\sum_i\alpha_i(p_i,w'_i)(v_i,q'_i)$ where
$\alpha_i\in K$, $p_i\in\cB, v_i\in \cQ_0,q'_i\in\cB'$, and $ w'_i\in \cQ'_0$. 
Next we apply the Fundamental Lemma and write, for each $i$,
$p_i= \iota_i+N_i$ and $q'_i=\iota'_i+N'_i$ where
$\iota_i\in I$, $N_i\in \Span_K(\cN)$ and
$\iota'_i\in I'$, $N'_i\in \Span_K(\cN')$.
Thus, 
\[X=\sum_i\alpha_i( \iota_i+N_i,w')(v_i,\iota'_i+N'_i)=
\sum_i\alpha_i[(\iota_i,\iota'_i)+(\iota,N_i')+(N_i,\iota'_i)+(N_i,N'_i)].\]
We are assuming that $\varphi(X)=0$.  Since $\iota_i\in I$ and $\iota'_i\in I'$,
we see that \[0=\varphi(X)=\varphi(\sum_i\alpha_i(N_i,N'_i))=\sum_i\alpha_iN_i\otimes N'_i.\]
Hence $0=\sum_i\alpha_iN_i\otimes N'_i$. 
Let $N_i=\sum_j\beta_{i,j}n_j$ and $N'_i=\sum\beta'_{i,j}n'_j$ with $n_j\in\cN$ and $n'_{j}\in\cN'$.  Then $\sum_i \alpha_i\beta_{i,j}\beta'_{i,j}=0$.
 This implies $\sum_i\alpha_i(N_i,N'_i)=0$.
Thus $X=\sum_i\alpha_i[(\iota_i,\iota'_i)+(\iota,N_i')+(N_i,\iota'_i)]$ and $X$ is generated
by $\cG^*$.

The elements of $\cG^*$ are clearly uniform.  Overlap relations  involving two elements
of the form $(g,w')$ with $g\in\cG$ completely reduce 0 using the complete reduction to 0
in $K\cQ$ by $\cG$.   Similarly,the overlap  relation of
two elements of the form $(v,g')$ with $g'\in\cG'$ completely reduce
to 0 by $\cG'$.  Now suppose $g=t - N\in\cG$ with $\tip(g)=t$.  Consider 
$Ov((u,b')(a,x'),( t,x'))$ where $ (u,b')(a,x')-(a,w')(v,b')\in  C$  and $t=a\hat t$ (with
$a\in\cQ_1, b\in\cQ'_1)$.
Then $Ov((u,b')(a,x'),(t,x'))=-(a,w')(v,b')(\hat t,x') + (u,b')(N,x')$.  Using simple reduction
involving elements of $C$, we commute $(u,b')$ past the elments of $\hat t$ and $N$ (changing
the vertices  as needed) to obtain that $Ov((u,b')(a,x'),(t,x'))$ reduces
to $-(a,w')(\hat t,w')(v,b') + (N,w')(v,b')$.  But $-(a,w')(\hat t,w')(v,b') + N(v,b')=-(a\hat t-N,w')
(v,b')=-(g,w')(v,b')$.  Since $(g,w)\in\cG^*$,   $Ov((u,b')(a,x'),(t,x'))$
completely reduces to 0.  The case of element of the form $(v,g')$, overlapping a commutitivity
relation, completely reducing to 0 is similar.  This completes the proof that $\cG$ is
a \grb basis for $I^*$ with respect to $\succ^*$.  We leave it to the reader to show that 
$\cG^*$ is the reduced
\grb basis.  This completes the proof of part 1.

The proof of part 2 is straightforward and left to the reader.
\end{proof}

The next result follows from Proposition \ref{prop-op} and Theorem \ref{thm-tensor}.

\begin{corollary}\label{cor-env}  Let $\Lambda=K\cQ/I$ be a strong Koszul algebra
(with respect to $I$ and $\succ$).  Let $\succ^*$ be the length admissible order
defined above where $\Lambda'=\Lambda^{op}$.  Then the enveloping algebra
$\Lambda\otimes_K\Lambda^{op}=K\cQ^*/I^*$ is a strong Koszul algebra  (with
respect to $I^*$ and $\succ^*$). \qed
\end{corollary}

\section{Remarks and questions}\label{sec-rem}  
The goal of this paper was to introduce the variety $\GrAlg(\cT)$ and its connection
to Koszul algebras.  We believe this connection will lead to further interesting
results, and, to that end, we present a number of open questions.  By dropping the restriction that $\cT$
is composed of paths of length 2, thus allowing $\cT$ to be an arbitrary
finite set of paths, one
still has a variety $\GrAlg(\cT)$ whose points correspond to certain graded algebras.  
In this more general setting, if one drops the length homogeneity restriction, one still obtains
an algebraic variety whose points correspond to (not necessarily graded) algebras.
Such connections are  currently under investigation in \cite{GHS}.

In the case of Koszul algebras that are not strong, one still can look at an appropriate
variety with $\cT$ being the set $\tip(\cG)$ where $\cG$ is a reduced \grb basis
with respect to some admissible  order. In such a variety, $K\cQ/\langle \cT\rangle$
is not a Koszul algebra since it is a nonquadratic monomial
algebra.  This leads to the following question:

\begin{Question}\label{qu-present}{\rm Suppose $\Lambda=K\cQ/I$ is a Koszul algebra
such that the reduced \grb basis $\cG$ is not a quadratic ideal. Then
$\Lambda_{Mon}=K\cQ/\langle \tip(\cG)\rangle$ is not a Koszul algebra.
Thus $\Lambda$ and $\Lambda_{Mon}$ are in $\GrAlg(\tip(\cG))$.   Are there necessarily
other Koszul algebras in $GrAlg(\tip(\cG))$ and, if so, does the set of
Koszul algebras in   $GrAlg(\tip(\cG))$ have some geometrical interpretation?
In particular, it would be interesting to study $GrAlg(\tip(\cG))$ for the
Sklyanin algebras. \qed
}
\end{Question}

\begin{Question}\label{qu-dual}{\rm If $\Lambda=K\cQ/I$ is a strong Koszul algebra, is
there a length admissible order $\succ^{\perp}$ on the paths in quiver $\cQ^{op}$
so that the Koszul dual, $\oplus_{n\ge 0}\Ext_{\Lambda}^n(\overline{\Lambda}.
\overline{\Lambda})$, is a strong Koszul algebra, where $\overline{\Lambda}$ is
$K\cQ/J$ with $J=\langle \cQ_1\rangle$?  We believe the answer is no.
\qed}\end{Question}

One can drop the graded restriction and allow arrows together with
paths of length two in a \grb basis.  The new variety contains $\GrAlg(\cT)$.
It would be interesting to study how these two varieties relate to one another. 

Basic geometic questions need to be investigated, some of which are listed below.  

\begin{Question}\label{qu-geo}{\rm
Is $\GrAlg(\cT)$ irreducible? Given Example \ref{ex-poly3-1} one expects
that the answer, in general, is no.  Example \ref{ex-poly2} shows that sometimes
$\GrAlg(\cT)$ is irreducible. What do the
irreducible components look like, and is there an interpretation
of irreducibility in terms of the algebras?
\qed
}\end{Question}

\begin{Question} {\rm Does every affine algebraic variety occur as some $\GrAlg(\cT)$? $\GrAlg_{\psi}(\cT)$?
}\end{Question}
\begin{Question}{\rm
What is the dimension of $\GrAlg(\cT)$ in terms of some invariant, like $\cT$?  Is there
an algorithm to compute it?
\qed}\end{Question}

\begin{Question}{\rm
Characterize when  $\GrAlg(\cT)$ is a point?
\qed}\end{Question}

\begin{Question}{\rm
Characterize when  $\GrAlg(\cT)$ is all of affine space?
\qed}\end{Question}

\begin{Question}\label{qu-self}{\rm  Suppose $\Lambda=K\cQ/I$ is a 
finite dimensional, selfinjective, strong
Koszul algebra (with respect to $I$ and $\succ$), does $\GrAlg(\tip(I))$ necessarily
contain other selfinjective, strong Koszul algebras?  Does the set of selfinjective
strong Koszul algebras have a geometrical interpretation in $\GrAlg(\tip(I))$?

Is there a geometric condition on $\GrAlg(\cT)$ that assures the existence of
a finite dimensional selfinjective algebra in $\GrAlg{\cT}$?
\qed}
\end{Question}

\begin{Question}\label{qu-bound}{\rm  Let $\cT$ be a set of paths of length 2 in
a quiver $\cQ$ and let $\succ$ be a length admissible order on the set
of  paths in $\cQ$.  Let $\cI$ be the ideal of the variety $\GrAlg(\cT)$.  Is there
a fast way to determine the degrees of a generating set  of  polynomials in
$\cI$ or a bound on their degrees?  More precisely,  overlap relations
result in paths  of length 3.  There is a bound on the number of simple reductions
needed to completely reduce to words  that are in $\cN_3$, i.e., nontips of length
3.   This number (plus 1)
should bound the degree of the polynomials in $\cI$.  Is there a fast way to find this number
given $\cT$?    
\qed
}\end{Question}

\end{document}